\documentclass[11pt, a4paper]{amsart}

\usepackage{amsmath,amssymb,amsfonts,mathrsfs,amsthm,bbm,enumitem, graphicx}

\DeclareMathOperator{\var}{var}

\theoremstyle{plain}
\newtheorem{theorem}{Theorem}
\newtheorem{lemma}{Lemma}
\newtheorem{prop}{Proposition}
\newtheorem{cor}{Corollary}

\theoremstyle{definition}

\newtheorem{defn}{Definition}

\newcommand{\N}{\mathbbm{N}}

\newcommand{\R}{\mathbbm{R}}

\DeclareMathOperator{\dimh}{dim_H}
\DeclareMathOperator{\dimp}{dim_P}
\DeclareMathOperator{\dimb}{dim_B}

\newcommand{\tyk}{\mathcal}
\def\B{\mathcal B}

\def\E{\tyk{E}_{\textup{ah}}}
\def\Aio{\tyk{A}_{\textup{i.o.}}}
\newcommand{\para}[1]{\left( #1 \right )}
\newcommand{\tub}[1]{\left\{#1\right\}}
\newcommand{\num}[1]{\left | #1\right |}

\newcommand{\Meng}[2]{\left\{\,#1\mathrel{}:\mathrel{}#2\,\right\}}

\begin{document}
\title[Shrinking targets \& eventually always hitting points]{Shrinking targets and eventually always hitting points for interval maps}

\author{Maxim Kirsebom}
\address{Maxim Kirsebom, University of Hamburg, Department of Mathematics, Bundesstrasse 55, 20146 Hamburg, Germany}
\email{maxim.kirsebom@uni-hamburg.de}

\author{Philipp Kunde}
\address{Philipp Kunde, Indiana University, Department of Mathematics, 831 East 3rd St., Bloomington, IN 47405, USA}
\email{pkunde@iu.edu}

\author{Tomas Persson}
\address{Tomas Persson, Centre for Mathematical Sciences, Lund University, Box 118, 221 00 Lund, Sweden}
\email{tomasp@maths.lth.se}

\thanks{We thank both J\"{o}rg Schmeling and Dmitry Kleinbock for
  interesting discussions and helpful comments and suggestions which
  improved the paper. We thank Jorge Freitas for answering our
  questions about Hitting Time Statistics for the Gau\ss{} map. We
  also acknowledge financial support by the Hamburg--Lund Funding
  Program 2018 which made several mutual research visits possible.}


\begin{abstract}
	We study shrinking target problems and the set $\E$ of eventually always hitting points. These are the points whose first $n$ iterates will never have empty intersection with the $n$-th target for sufficiently large $n$. We derive necessary and sufficient conditions on the shrinking rate of the targets for $\E$ to be of full or zero measure especially for some interval maps including the doubling map, some quadratic maps and the Manneville--Pomeau map. We also obtain results for the Gau\ss{} map and correspondingly for the maximal digits in continued fractions expansions. In the case of the doubling map we also compute the packing dimension of $\E$ complementing already known results on the Hausdorff dimension of $\E$.
	\end{abstract}

\maketitle

\section{Introduction and setup}

The term \emph{shrinking target problems} in dynamical systems describes a class of questions which seek to understand the recurrence behaviour of typical orbits of a dynamical system. The standard ingredients of such questions are a measure-preserving dynamical system $(X,\mu, T)$, $T \colon X\to X$ and $\mu$ being a finite measure. Also, we have a sequence of subsets $\tub{B_m}_{m=1}^{\infty}$ with $B_m\subset X$ and $\mu(B_m)\to 0$. Recently shrinking target problems have also been investigated in the case when the measure is infinite, see \cite{GalatoloHollandPerssonZhang}.

In this paper we focus mostly on the case of finite measure and whenever this is the case we assume the measure to be normalized to a probability measure. If nothing else is stated this may be assumed to be the setting. A few of our results concern infinite measures and it will be stated explicitly whenever this is the case. 

Throughout this paper $(X,T,\mu)$ will
always denote a measure preserving system and
$\mathbf{B}:=\tub{B_m}_{m=1}^{\infty}$ will always denote a sequence
of subsets of $X$ for which $\mu(B_m)\to 0$. We refer to this as a
\emph{sequence of shrinking targets}. We call the sequence
\emph{nested} if $B_m\supset B_{m+1}$ for all $m$.

Classical questions in this area focus on the set of points in $X$, whose $n$'th iterate under $T$, lies in the set $B_n$ for infinitely many $n$. That is, given a sequence $\mathbf{B}=\tub{B_m}_{m=1}^{\infty}$
\begin{equation*}
	\Aio=\Aio(\mathbf{B}):=\tub{x\in X:T^{n}x\in B_n \text{ for infinitely many } n\in\N}.
\end{equation*}
If $\sum \mu(B_m)<\infty$ the Borel--Cantelli Lemma tells us that
$\mu(\Aio)=0$. If $\sum\mu(B_m)=\infty$ the situation is more
complicated since the Borel--Cantelli Lemma only guarantees
$\mu(\Aio)=1$ for independent events and this is usually not
satisfied for dynamical systems. If we do have a sequence $\mathbf{B}$ for
which $\mu(\Aio(\mathbf{B}))=1$ then we call $\mathbf{B}$ a \emph{Borel--Cantelli
  \emph{(BC)} sequence}. If we can prove that a large family $\B$ of
sequences are all BC-sequences then we say that we have a
\emph{dynamical Borel--Cantelli Lemma}. In many cases such lemmas hold
if the system satisfies some version of mixing which essentially acts
as a replacement for independence. However, it is known that for any
measure-preserving system we can find a sequence $\mathbf{B}$ satisfying
$\sum\mu(B_m)=\infty$ which is not BC for the system. We may even find
such a sequence $\mathbf{B}$ which is nested \cite[Proposition~1.6]{ChernovKleinbock}. Hence there is no hope for $\B$ to be all
sequences satisfying $\sum\mu(B_m)=\infty$. It is therefore natural to
look for the largest possible sub-families of $\B$ which consist only
of BC-sequences. It turns out that sequences of balls with fixed
center, and nested sequences of balls with fixed center are good and
natural candidates. We say that $(X,\mu,T)$ has the \emph{shrinking target
  property} (STP) if for any $x_0\in X$, every sequence of balls $B_m$
centered at $x_0$ satisfying $\sum\mu(B_m)=\infty$ is a
BC-sequence. We say that $(X,\mu,T)$ has the \emph{monotone shrinking
  target property} (MSTP) if for any $x_0\in X$, every nested sequence
of balls $B_m$ centered at $x_0$ satisfying $\sum\mu(B_m)=\infty$ is a
BC-sequence. Many interesting systems are known to have either the STP
or MSTP property, see \cite{Athreya}, \cite{Tseng}, and references
therein for examples. A more comprehensive introduction
to dynamical Borel--Cantelli lemmas, including examples, can be found
in \cite{ChernovKleinbock}.

In this paper we are interested in similar properties for a certain subset (aside for a set of measure 0) of $\Aio$ known as the set of points which are \emph{eventually always hitting}. Due to the central importance of this concept we introduce it through a separate definition.
\begin{defn}[Eventually always hitting]
	A point $x\in X$ is said to be \emph{eventually always hitting} (EAH) for $\mathbf{B}=\tub{B_m}_{m=1}^{\infty}$ under $T$ if there exists some $m_0(x)\in\N$ such that for all $m\geq m_0(x)$ we have
	\begin{equation*}
		\tub{x, T(x), T^2 (x),\dots, T^{m-1}(x)}\cap B_m \neq \emptyset.
	\end{equation*}
	The set of all points in $X$ which are eventually always hitting for $\mathbf{B}$ under $T$ will be denoted $\E:=\E(\mathbf{B})$.
\end{defn}

We remark that some authors study a slightly different version of
eventually almost hitting, and require that for all $m \geq m_0 (x)$
we have $T^k (x) \in B_m$ for some $k \leq m$, whereas we require $0\leq k <
m$. For the results that we are discussing in this paper, it is
unimportant which definition we use. The results are the same, with
the same proofs, if we use the other definition instead.

Note that $\E$ may also be written as
\begin{align*}
\E 	&= \Bigg\{x \in X : \begin{array}{l}\exists\, m_0(x)\in\N\; \forall\, m > m_0 (x) \;\\ \exists\, k\in\tub{0,\dots,m-1}\text{ s.t. } T^k (x) \in B_m\end{array}\Bigg\} \\
	&= \bigcup_{n=1}^{\infty} \bigcap_{m =n}^\infty \bigcup_{k=0}^{m-1} T^{-k} (B_m).
\end{align*}
Set
\begin{equation*}
	\Lambda:=\bigcup_{k=0}^{\infty} T^{-k}\para{\bigcap_{i=1}^{\infty} B_i}.
\end{equation*}
Then $\mu(\Lambda) = 0$ since $\mu (B_i) \to 0$, and we have that
$\E\backslash\Lambda\subset \Aio$. In this sense, being eventually
always hitting for $B_m$ is a stronger property than hitting $B_m$
infinitely often.

The term eventually always hitting was coined by Kelmer in \cite{Kelm}
where this set was studied in the context of flows on hyperbolic
manifolds. Kelmer proved necessary and sufficient conditions for the set of eventually always hitting points to be of full measure. Shortly afterwards Kelmer and Yu \cite{KeYu} extended the
investigation to flows on homogeneous spaces. Also, Kleinbock and
Wadleigh \cite{KleinWad} studied the concept in the context of higher
dimensional Diophantine approximations. Imposing a long-term independence property on the shrinking target system Kleinbock, Konstantoulas and Richter \cite{KKR} recently obtained tight conditions on the shrinking rate of the targets so that $\E$ has measure zero or full measure. In particular, their assumptions are satisfied for specific choices of targets in product systems and Bernoulli shifts.

However, the concept had
already been considered a few years earlier by Bugeaud and Liao
\cite{BugeaudLiao} for a particular sequence of targets with
exponential rate of shrinking in the setting of
$\beta$-transformations $T_{\beta}(x) = \beta x - \lfloor \beta
x\rfloor$ on $[0,1]$ for every $\beta >1$. For $x \in [0,1]$ they
introduce the exponent $\hat{\nu}_{\beta}(x)$ as the supremum of real
numbers $\hat{\nu}$ for which for every sufficiently large $N \in \N$
the inequality $T^n_{\beta}(x) < \left(\beta^N\right)^{- \hat{\nu}}$
has a solution $1 \leq n \leq N$. Note that this corresponds to $x$
satisfying $\{T_{\beta}^n(x)\}_{n=1}^N\cap
B(0,\beta^{-\hat{\nu}N})\neq \emptyset$ for every sufficiently large
$N$. Hence the set $\{x\in [0,1]:\hat{\nu}_{\beta}(x)\geq \hat{\nu}\}$
corresponds to $\E(B(0,\beta^{-\hat{\nu}N}))$ in our notation (aside
from the discrepancy in definition mentioned above). They show that
\[
\dimh \left( \Meng{x}{ \hat{\nu}_{\beta}(x) \geq \hat{\nu} }\right) = \Bigl( \frac{1-\hat{\nu}}{1+\hat{\nu}}\Bigr)^2
\]
for every $\hat{\nu} \in [0,1]$, where $\dimh$ is the Hausdorff dimension. Bugeaud and Liao also obtain analogous results in the setting of $b$-ary expansions.

In this paper we prove various results concerning the measure and also dimension of $\E$ in different settings. Our different results apply in various levels of generality, hence it would be complicated to state them all accurately in this introduction. Instead we illustrate our results through application to specific systems with simpler assumptions.

Note that throughout the paper $\log$ will denote the natural logarithm. 

\subsection{Main results for specific systems}

\begin{theorem}[The doubling map]
	\label{thm1}
	Let $X=[0,1]$, let $T \colon X\to X$ be the doubling map $Tx=2x\mod 1$ and let $\mu$ denote the Lebesgue measure on $X$. Let $\tub{B_m}_{m=1}^{\infty}$ denote a nested sequence of shrinking intervals with fixed center.
	\begin{enumerate}[label=(\alph*)]
		\setlength\itemsep{0.8em}
		\item\label{thm1a} If $\mu(B_m)\leq\frac{c}{m}$ for some $c\in\R$, then $\mu(\E)=0$.
		\item\label{thm1b} If $\mu(B_m)\geq \frac{c (\log m)^2}{m}$ for some $c>0$ sufficiently large, then $\mu(\E)=1$.
	\end{enumerate}
Let $B_m=B(0,2^{-sm})$, $s\in\R$ and $\dimp$ denote the packing dimension.
	\begin{enumerate}[resume, label=(\alph*)]
		\setlength\itemsep{0.8em}
		\item\label{thm1c} If $s\in (0,1)$, then $\dimp\E=1-s$.
		\item\label{thm1d} If $s\geq 1$ then $\E$ is countable and in particular $\dimp\E=0$.
		\end{enumerate}
\end{theorem}

We remark that we prove the statements \ref{thm1a} and \ref{thm1b} in
the theorem above for many other dynamical systems, for instance piecewise expanding maps and some quadratic maps. See Section~\ref{section4} for more details.

Note also that for the doubling map, and $B_m = B(0,2^{-sm})$, the result
of Bugeaud and Liao \cite{BugeaudLiao} implies that $\dimh \E = \bigl(
\frac{1-s}{1+s} \bigr)^2$. Hence we have
\[
0 < \dimh \E = \Bigl(\frac{1-s}{1+s} \Bigr)^2 < \dimp \E = 1 - s < 1
\]
when $s \in (0,1)$.

\begin{theorem}[The Manneville--Pomeau map]
  \label{thmMP}
  Let $X=[0,1]$, let $\alpha > 0$ and let $g_{\alpha}\colon[0,1)\to
    [0,1)$ be the Manneville--Pomeau map given by
  \begin{equation}\label{MPmap}
    g_{\alpha}(x)=
    \begin{cases} x(1+2^{\alpha}x^{\alpha}) & \text{if } x\in [0,\frac12)\\
        2x-1 & \text{if } x\in [\frac12,1).
    \end{cases}
  \end{equation}
  Let $\mu_{\alpha}$ denote the absolutely continuous invariant
  measure for $([0,1),g_{\alpha})$, which is finite if and only if
  $\alpha \in (0,1)$. Let $\tub{B_m}_{m=1}^{\infty}$ denote a
  nested sequence of intervals with a fixed center
  which is not $0$.

  If $\alpha \in (0,1)$ then we have the following results.
  \begin{enumerate}[label=(\alph*)] 
    \setlength\itemsep{0.8em}
  \item\label{thmMPa} If $\mu_\alpha (B_m) \leq \frac{c}{m}$ for some constant $c$,
    then $\mu_\alpha (\E)=0$.
  \item\label{thmMPb} If $\mu_\alpha (B_m) \geq \frac{c (\log m)^{2 +
      \varepsilon}}{m}$ for some $\varepsilon > 0$ and any $c>0$, then
    we have $\mu_\alpha (\E)=1$.
  \end{enumerate}

  If $\alpha \geq 1$ then we have the following results.
  \begin{enumerate}[resume, label=(\alph*)]
    \setlength\itemsep{0.8em}
  \item\label{thmMPc} If $\alpha>1$ and $\mu_\alpha (B_m)\leq\frac{c}{m^{\frac {1}{\alpha} + \varepsilon}}$
    for some $\varepsilon > 0$ and any $c > 0$, then $\mu_\alpha (\E)=0$. 
    
    \noindent If $\alpha=1$ and $\mu_{\alpha}(B_m)\leq \frac{c}{m}$ for some $c>0$ then $\mu_\alpha (\E)=0$. 
  \item\label{thmMPd} If $\mu_\alpha (B_m) \geq \frac{c}{m^{\frac{1}{\alpha} -
      \varepsilon}}$ for some $\varepsilon > 0$ and any $c>0$, then
    $\mu_\alpha (\complement \E)=0$.
  \end{enumerate}
\end{theorem}

We remark that the measure $\mu_\alpha$ is equivalent to Lebesgue
measure on $[0,1]$. Hence the statements $\mu_\alpha (\E) = 0$ and
$\mu_\alpha (\complement \E) = 0$ are equivalent to the corresponding
statements involving the Lebesgue measure instead. Note also that in
the case $\alpha \in (0,1)$ after normalising the measure
$\mu_\alpha$ we can state $\mu_\alpha (\complement \E) = 0$ as the
equivalent statement $\mu_\alpha (\E) = 1$. This is not possible if
$\alpha \geq 1$, since the measure $\mu_\alpha$ is not finite in this
case.

We may also consider the Gau\ss{} map defined by $G:(0,1]\to [0,1)$
given by $G(x)=\frac{1}{x}\mod 1$. $G$ admits an absolutely continuous
invariant probability measure known as the Gauss measure which has
density $\frac{1}{\log 2}\frac{1}{1+x}$. Statement \ref{thm1b} of
Theorem~\ref{thm1} also holds true for the Gau\ss{} map and measure,
while \ref{thm1a} of Theorem~\ref{thm1} holds true in this setting
when $B_m:=[0,r_m)$. This allows us to obtain a statement about the
eventually always hitting property for maximal digits of continued
fractions expansions. Recall that every point $x\in (0,1]$ can be
written as a continued fraction, i.e.
\begin{align*}
x=\cfrac{1}{a_1(x)+\cfrac{1}{a_2(x)+\cfrac{1}{a_3(x)+ \cfrac{1}{\ddots}}}}
\end{align*}
where the $a_i(x)$'s are generated by the algorithm $a_1(x)=\lfloor\frac{1}{x}\rfloor$ and $a_j(x)=a_1(G^{j-1}(x))$. In compact notation we write $x=[a_1(x),a_2(x),\dots]$.

\begin{theorem}[Continued fractions]
	\label{ThmCF}
	Let $\mu$ denote the Gau\ss{} measure.
	\begin{enumerate}[label=(\alph*)]
		\setlength\itemsep{0.8em}
		\item\label{thmCFa} For any $c > 0$ we have
		\begin{equation*}
		  \mu \Bigl( \Bigl\{\, x\in
                      [0,1]:\exists\,m_0\in\N\,\forall m\geq
                      m_0:\max_{1\leq k\leq m} a_k(x)\geq cm \,
                      \Bigr\} \Bigr) =0.
		\end{equation*}
		\item\label{thmCFb} If $c > 0$ is sufficiently small,
                  then
		\begin{equation*}
		  \mu \biggl( \biggl\{\, x\in
                      [0,1]:\exists\,m_0\in\N\,\forall m\geq
                      m_0:\max_{1\leq k\leq m} a_k(x)\geq \frac{cm}{(\log m)^2}
                      \,\biggr\} \biggr) =1.
		\end{equation*}
		\end{enumerate}
\end{theorem}

\subsection{Structure of the paper} 

In Section~\ref{section2} we give some preliminary, general results
concerning the set of eventually always hitting points that will prove
useful later in the paper. In Section~\ref{section3} we give necessary
and sufficient conditions for $\E$ to be of full measure. In
Section~\ref{section4} we cite various results on mixing and hitting
time statistics which in conjunction with the results of
Section~\ref{section3} allows us to deduce the conclusions of
Theorem~\ref{thm1}, \ref{thmMP} and \ref{ThmCF} (except
Theorem~\ref{thm1}\ref{thm1c} and \ref{thm1d}). We also discuss
further systems for which the results of Section~\ref{section3} can be
applied. Section~\ref{section5} is dedicated to the proof of
Theorem~\ref{thm1}\ref{thm1c} and \ref{thm1d}.

\section{Preliminaries on eventually always hitting points}
\label{section2}

Recall that
\begin{equation*}
\E = \bigcup_{n=1}^{\infty} \bigcap_{m =n}^\infty \bigcup_{k=0}^{m-1} T^{-k} (B_m).
\end{equation*}
It will prove convenient to write
\begin{equation*}
\E =\bigcup_{n=1}^{\infty} A_n	
\end{equation*}
where
\begin{equation*}
A_n =\bigcap_{m = n}^\infty C_m \quad , \quad	C_m = \bigcup_{k=0}^{m-1} T^{-k} (B_m).
\end{equation*}
Note in particular that $A_n\subset A_{n+1}$ and hence
\begin{equation*}
\mu\para{\E}=\lim_{n\to\infty}\mu(A_n).
\end{equation*}
The following lemma will prove very useful for studying the measure of $\E$ when the sets $B_m$ are assumed to be nested. It states that $\E$ is an almost invariant set under $T$. More precisely, this means that the symmetric difference between $\E$ and $T^{-1}(\E)$ is a set of measure zero.

\begin{lemma}\label{Lemma1}
	Let $(X, \mu, T)$ be a measure-preserving dynamical system ($\mu$ either finite or infinite) and let $\{B_m\}_{m=1}^{\infty}$ be a nested family of shrinking targets, i.e.\ $B_m\supset B_{m+1}$ and $\mu (B_m) \to 0$. Then
	\[
	\mu \left( \E \triangle T^{-1}\left( \E \right) \right) = 0. 
	\]
\end{lemma}

Hence, for ergodic transformations with respect to a finite $\mu$, $\E$ obeys a zero--one law. That is, either $\mu \left( \E\right) = 0$ or $\mu \left( \E\right) = 1$. For ergodic transformations with respect to an infinite $\mu$, $\E$ obeys a zero--infinity law.

\begin{proof}
	Let $x \in \E$. Then there is $m_0:=m_0(x) \in \N$ such that for all $m \geq m_0$ there is $k \in \{ 0,1,\dots,m-1\}$ with $T^k(x) \in B_m$. Actually, if $x \notin \bigcap_{m \in \N} B_m$ there is even $\tilde{m}_0:= \tilde{m}_0(x) \in \N$ such that for all $m \geq \tilde{m}_0$ there is $k \in \{1,\dots,m-1\}$ with $T^k(x) \in B_m$ (because otherwise $x=T^0(x)$ would have to be in $B_m$ for all $m$ due to $x \in \E$ and the nesting property). Since the targets are nested, we also get for all $k \neq 0$ that
	\[
	T^{k-1}(T(x)) = T^k(x) \in B_m \subseteq B_{m-1}.
	\]
	Hence, $T(x) \in \E$. So, $\E \setminus \bigcap_{m \in \N} B_m \subseteq T^{-1}\left( \E \right)$. Since $\mu\left(\bigcap_{m \in \N} B_m \right)$ is a set of measure zero by $\mu(B_m) \to 0$ and $T$ is measure-preserving, we have
	\[
	\mu \left( \E \right) = \mu \left( \E\setminus \bigcap_{m \in \N} B_m \right) \leq \mu\left( T^{-1}\left( \E \right) \right) = \mu \left( \E \right),
	\]
	which yields the claim.
\end{proof}

In \cite{Kelm} Kelmer gave the following simple conditions for $\E$ to be of measure zero or one. We repeat the proof for completeness. Note that no assumption is made on the shape of the target sets.
\begin{prop}\label{Prop1} Let $\tub{B_m}_{m=1}^{\infty}$ denote a sequence of shrinking targets in $X$. 
	  \begin{enumerate}[resume, label=(\alph*)]
	  	\item\label{prop1a} Let $\mu$ be a probability measure and assume that $\mu(\E)=1$. Then there exists a sequence $c_m\to 1$ such that $\mu(B_m)\geq \frac{c_m}{m}$.
	  	\item\label{prop1b} Let $\mu$ be an infinite measure and assume that $\mu(\E)=\infty$. Then there exists a sequence $c_m\to \infty$ such that $\mu(B_m)\geq \frac{c_m}{m}$.
	  	\end{enumerate}
	\begin{proof}\ref{prop1a}
		By the assumption we get that
		\begin{align*}
		  \mu\para{\bigcup_{n=1}^{\infty}
                    A_n}=\lim_{n\to\infty}\mu(A_n)=\lim_{n\to\infty}\mu\para{\bigcap_{m=n}^{\infty}
                    C_m}=1
		\end{align*}
		which implies that $\mu(C_m)\to 1$ as $m\to\infty$. Now,
		\begin{align*}
		  \mu(C_m)=\mu\para{\bigcup_{k=0}^{m-1}
                    T^{-k}(B_m)}\leq \sum_{k=0}^{m-1}\mu \bigl(
                  T^{-k}(B_m) \bigr) =m\mu(B_m)
		\end{align*}
		where we used the $T$-invariance of $\mu$. Hence we see that since $\mu(C_m)\to 1$, we have $\mu(B_m)\geq \frac{c_m}{m}$ for some sequence $c_m \to 1$, because otherwise there is an $\varepsilon > 0$ and infinitely many $m$ for which $\mu (B_m) \leq \frac{1-\varepsilon}{m}$, which contradicts the estimates above.		
		
		\noindent \ref{prop1b} The proof goes exactly like for \ref{prop1a} with the obvious adaptations.  
	\end{proof}
\end{prop}

For nested sequences $\tub{B_m}_{m=1}^{\infty}$ and $T$ being ergodic we get the following
sufficient condition for $\mu(\E)=0$. The proof is just the negation
of Proposition~\ref{Prop1} followed by an application of
Lemma~\ref{Lemma1}.

\begin{cor} \label{cor:measure0}
  Assume that $T$ is ergodic and that $\tub{B_m}_{m=1}^{\infty}$ is a nested sequence of
  shrinking targets. 
  	  \begin{enumerate}[resume, label=(\alph*)]
  	  	\item\label{cor1a} Let $\mu$ be a probability measure. If there exists a $c<1$ such that $\mu(B_m)\leq \frac{c}{m}$ holds for infinitely many $m$, then $\mu(\E)=0$.
  	  	\item\label{cor1b} Let $\mu$ be an infinite measure. If there exists a $c\in\R$ such that $\mu(B_m)\leq \frac{c}{m}$ holds for infinitely many $m$, then $\mu(\E)=0$.
	\end{enumerate}
\end{cor}

\section{Necessary and sufficient conditions for $\mu(\E)=1$}
\label{section3}

In this section we give proofs of various new necessary and sufficient
conditions for $\mu(\E)$ to be of measure zero or one.

\subsection{Necessary conditions for $\boldsymbol{\mu(\E)=1}$}

We introduce some terminology and notation about hitting times in dynamical system. Given a set $E\subset X$, we denote by $\tau_E \colon X\to \N$ the \emph{first hitting time to $E$} which is defined by
\begin{align*}
\tau_E (x)=\inf\tub{i\in\N:T^i x\in E}.
\end{align*}
Let $E_m\subset X$ denote a sequence of sets for which $\mu(E_m)\to 0$ and define the function
\begin{align*}
  G_{E_m} (t):=\limsup_{m \to \infty}\mu\para{\tub{x\in X:
      \tau_{E_m}(x)\leq\frac{t}{\mu(E_m)}}}.
\end{align*}
The next easy proposition gives a necessary condition for $\E$ to be of full measure when $G(t)<1$ for all $t\in\R$. 
\begin{prop} Let $\tub{B_m}_{m=1}^{\infty}$ denote a sequence of shrinking targets and assume that $G_{B_m} (t)<1$ for all $t\in\R$.
If $\mu(\E)=1$, then for every $c\in\R$ we have $\mu(B_m)\geq \frac{c}{m}$ for all sufficiently large $m\in\N$.
\end{prop}

\begin{proof}
  To get a contradiction, fix $c\in\R$, and assume that there is a sequence $m_j \to \infty$
  such that $\mu(B_{m_j}) \leq \frac{c}{m_j}$ for all $j$. We then have
  \[
  \tub{x\in X: \tau_{B_{m_j}}(x)\leq m_j} \subset \tub{x\in X:
    \tau_{B_{m_j}}(x)\leq \frac{c}{\mu(B_{m_j})}}.
  \]
  Using this inclusion we may rewrite as follows
  \begin{align*}
    \E&=\bigcup_{n=1}^{\infty}\bigcap_{m=n}^{\infty} \tub{x\in X:
      \tub{T^i x}_{i=0}^{m-1} \cap B_m\neq
      \emptyset}\\ &=\bigcup_{n=1}^{\infty}\bigcap_{m=n}^{\infty}
    \tub{x\in X: \tau_{B_m}(x)< m}\\ &\subset
    \bigcup_{n=1}^{\infty}\bigcap_{m_j \geq n} \tub{x\in X:
      \tau_{B_{m_j}}(x)\leq m_j} \subset
    \bigcup_{n=1}^{\infty}\bigcap_{m_j \geq n} \tilde{C}_{m_j},
  \end{align*}
  where
  \[
  \tilde{C}_{m_j} = \tub{x\in X: \tau_{B_{m_j}}(x)\leq \frac{c}{\mu(B_{m_j})}}.
  \]
  Assuming that $\mu(\E)=1$, we argue as in Proposition~\ref{Prop1}
  that we must have $\mu(\tilde{C}_{m_j})\to 1$ which means that
  \begin{equation*}
    \mu\para{\tub{x\in X: \tau_{B_{m_j}}(x)\leq
        \frac{c}{\mu(B_{m_j})}}}\to 1
  \end{equation*}
  for $m_j \to\infty$. However, $\limsup_{j \to \infty}
  \mu(\tilde{C}_{m_j})\leq G_{B_m} (c)<1$ by assumption and hence we
  have a contradiction. Since this inequality is true for all $c\in\R$
  we get the desired conclusion.
\end{proof}

Again we get a sufficient condition for $\mu(\E)=0$ under additional assumptions.
\begin{cor} \label{cor:measure0RTS}
  Assume that $T$ is ergodic and that $\tub{B_m}_{m=1}^{\infty}$ is a nested sequence of
  shrinking targets and assume that $G_{B_m} (t)<1$ for all $t\in\R$. If there exists $c\in\R$ such that
  $\mu(B_m)\leq \frac{c}{m}$ for infinitely many $m$, then
  $\mu(\E)=0$.
\end{cor}

The condition $G_{B_m} (t)<1$ is easily satisfied in many cases. It is
interesting to note that often much more is known about $G_{B_m}
(t)$. If we assume $B_m$ to be a nested sequence of shrinking balls
with fixed center it is often known that
\begin{equation*}
  G_{B_m} (t)=\lim_{m\to\infty}\mu\para{\tub{x\in X:
      \tau_{B_m}(x)\leq\frac{t}{\mu(B_m)}}}
\end{equation*}
exists and is non-degenerate which means that $G_{B_m}(t)$ takes at
least one value different than 0 or 1. This property is known as the
system having \emph{hitting time statistics} (HTS) to $B_m$. Among
these systems, many have exponential HTS to $B_m$ meaning that
$G_{B_m} (t)=1-e^{-t}$. It is from the rich theory of HTS for
dynamical systems that we borrow in order to prove
Theorem~\ref{thm1}\ref{thm1a}, \ref{thmMP}\ref{thmMPa} and \ref{ThmCF}\ref{thmCFa}. We elaborate on this point in
Section~\ref{section4}.

\subsection{Sufficient conditions for $\boldsymbol{\mu(\E)=1}$}

Let $X\subset\R$ in our probability measure preserving system
$(X,T,\mu)$ and let $B_m = B(y_m, r_m)$ be a sequence of balls in $X$.
We consider the $L^1$ and $BV$ norms of functions on $f \colon X \to \R$, defined by
\begin{align*}
  \lVert f \rVert_1 &= \int |f| \, \mathrm{d} \mu, \\
  \lVert f \rVert_{BV} &= \var f + \lVert f \rVert_1,
\end{align*}
where $\var f$ denotes the total variation of $f$.

We say that correlations decay as $p \colon \N \to \R$
for $L^1$ against $BV$, if
\begin{equation} \label{eq:correlation}
\biggl| \int f \circ T^n g \, \mathrm{d} \mu - \int f \, \mathrm{d}
\mu \int g \, \mathrm{d} \mu \biggr| \leq \lVert f \rVert_1 \lVert g
\rVert_{BV} p (n)
\end{equation}
holds for all $n$ and all functions $f$ and $g$ with $\lVert f \rVert_1$, $\lVert g \rVert_{BV}<\infty$.

\begin{theorem} \label{the:sufficient}
	Suppose that correlations decay as $p$ for $L^1$ against $BV$.
	
	If $p$ satisfies
	\begin{equation}\label{expmix}
	p(n) \leq C e^{- \tau n}
	\end{equation}
	for some $\tau > 0$, then $\mu (\E) = 1$
	provided
	\[
	\mu (B_m) \geq \frac{c (\log m)^2}{m}
	\]
	for some $c > \tau^{-1}$ and all sufficiently large $m$.  In
	particular $\mu (\E) = 1$ provided $\mu (B_m)
	\geq \frac{c (\log m)^2 h(m)}{m}$ for any $c > 0$ and any function $h$ for which $h(m)\to\infty$ as $m\to\infty$.
	
	If $p$ satisfies
	\begin{equation}\label{polmix}
	p(n) \leq \frac{C}{n^t}
	\end{equation}
	for some $t > 0$, then $\mu (\E) = 1$
	provided
	\[
	\mu (B_m) \geq \frac{c}{m^a}
	\]
	for any $c > 0$ and $a < \frac{t}{1+t}$.
\end{theorem}

Note that do not require that the balls are nested in Theorem~\ref{the:sufficient}.
We will need the following lemma.

\begin{lemma} \label{lem:higherordercorrelation}
	Suppose that \eqref{eq:correlation} holds for all $n$ and functions
	$f$ and $g$. If $f_k \colon X \to [0, \infty)$ and $n_k \in
	\N$, then
	\begin{align*}
	\int \prod_{k=1}^n f_k \circ T^{n_1 + \cdots + n_k} \, \mathrm{d} \mu &\leq \int f_n \, \mathrm{d} \mu \prod_{k=1}^{n-1} \biggl( \int f_k \, \mathrm{d} \mu + p(n_{k+1}) \lVert f_k \rVert_{BV} \Biggr)\\
	&\leq \prod_{k=1}^{n} \biggl( \int f_k \, \mathrm{d} \mu + p(n_{k+1}) \lVert f_k \rVert_{BV} \Biggr).
	\end{align*}
\end{lemma}

\begin{proof}
	Note that since $f_k \geq 0$, we have $\int f_k \, \mathrm{d} \mu =
	\lVert f_k \rVert_1$.  By \eqref{eq:correlation}, we have
	\[
	\int \prod_{k=1}^n f_k \circ T^{n_1 + \cdots + n_k} \, \mathrm{d}
	\mu \leq \biggl( \int f_1 \mathrm{d} \mu + p(n_2) \lVert f_1 \rVert_{BV}
	\biggr) \int \prod_{k=2}^n f_k \circ T^{n_1 + \cdots + n_k} \,
	\mathrm{d} \mu,
	\]
	and the first inequality follows by induction. The second inequality then follows trivially.
\end{proof}

\begin{proof}[Proof of Theorem~\ref{the:sufficient}]
	Recall that
	\[
	\E = \bigcup_{n=1}^{\infty} A_n,
	\]
	We prove that $\mu (A_n) \to 1$ as $n
	\to \infty$ which is equivalent to $\mu (\complement A_n) \to 0$. We
	can write $\complement A_n$ as
	\[
	\complement A_n = \bigcup_{m=n}^\infty D_m, \qquad \text{where } D_m
	= \bigcap_{k=0}^{m-1} T^{-k} (\complement B_m).
	\]
	Hence, it is sufficient to prove that
	\begin{equation} \label{eq:summable}
	\sum_{m=1}^\infty \mu (D_m) < \infty.
	\end{equation}
	
	We will now bound $\mu (D_m)$ from above. Let $\Delta_m > 1$. We
	have that
	\[
	\mu (D_m) \leq \mu (\tilde{D}_m), \qquad \text{where } \tilde{D}_m =
	\bigcap_{1 \leq k \leq m/\Delta_m} T^{-\Delta_m k} (\complement
	B_{m}).
	\]
	Then, we have by Lemma~\ref{lem:higherordercorrelation} that
	\begin{align*}
	  \mu (\tilde{D}_m) & = \int \prod_{1 \leq k \leq m/\Delta_m}
          1_{\complement B_{m}} \circ T^{\Delta_m k} \,
          \mathrm{d} \mu \\ &\leq \prod_{1 \leq k \leq m/\Delta_m}
          \biggl( \mu (\complement B_{m}) + p(\Delta_m)
          \lVert 1_{\complement B_{m}} \rVert_{BV} \biggr).
	\end{align*}
	Since $B_m$ are balls, we have $\var 1_{\complement B_{m}}
	\leq 2$ and $\lVert 1_{\complement B_{m}} \rVert_1 \leq
	1$. Hence we have $\lVert 1_{\complement B_{m}} \rVert_{BV} \leq
	3$ and we get
	\begin{align*}
	  \mu (\tilde{D}_m) &\leq \prod_{1 \leq k \leq m/\Delta_m }
          \biggl(1 - \mu (B_{m}) + 3 p(\Delta_m) \biggr) \\ &
          = \exp \biggl( \sum_{1 \leq k \leq m /\Delta_m} \log \bigl(1
          - \mu (B_{m}) + 3 p (\Delta_m) \bigr) \biggr) \\ &
          \leq \exp \biggl( \frac{m}{\Delta_m} 3p(\Delta_m) -\sum_{1
            \leq k \leq m/\Delta_m} \mu (B_{m}) \biggr),
	\end{align*}
	where the last inequality holds since $\log (1+x) \leq x$.
	
	Assume now that $p (n) \leq C e^{-\tau n}$. We prove that $\mu
	(\E) = 1$ if for all sufficiently large $m$, we
	have $\mu (B_m) \geq \frac{c (\log m)^2}{m}$ where $c > \tau^{-1}$.
	
	Take $\Delta_m = K \log m$ with any $K > \tau^{-1}$. Then
	\[
	\frac{m}{\Delta_m}3p(\Delta_m) \to 0 \qquad
	\text{as } m \to \infty,
	\]
	so we can assume $(m/\Delta_m)3p(\Delta_m)<1$ if $m$ is
	large.
	
	We assume that we have $\mu (B_m) \geq \frac{c (\log m)^2}{m}$ for all
	$m$. The proof also works with obvious changes if this is only the
	case for all large enough $m$. We have
	\[
	  \sum_{1 \leq k \leq m/\Delta_m} \mu (B_{m}) \geq
          \frac{m}{\Delta_m} \frac{c (\log m)^2}{m} = \frac{c}{K}
          \log m.
	\]
	
	Taken together, these two estimates give us the estimate
	\[
	\mu (\tilde{D}_m) \leq \exp \Bigl(1 - \frac{c}{K} \log m
        \Bigr),
	\]
	which implies that $\mu (\tilde{D}_m)$ is summable if $\frac{c}{K}
	> 1$. Since we may take $K$ as close to $\tau^{-1}$ as we like, we
	can make $\frac{c}{K} > 1$ as long as $c > \tau^{-1}$.
	
	We now consider the case that $p(n) \leq C/n^t$ and $\mu (B_m) \geq
	\frac{c}{m^a}$. Let $\Delta_m = m^\frac{1}{1 + t}$. Then
	\[
	\sum_{1 \leq k \leq m/\Delta_m} 3 p (\Delta_m) \leq \frac{m}{\Delta_m}
	\frac{3C}{\Delta_m^t} = 3C.
	\]
	Hence
	\[
	\mu (\tilde{D}_m) \leq \exp \biggl(3C - \sum_{1 \leq k \leq
          m^\frac{t}{1+t}} \frac{c}{m^a} \biggr) = \exp \Bigl( 3C - c
        m^\beta \Bigr),
	\]
	where $\beta = \frac{t}{1+t} - a$. If $a < \frac{t}{1+t}$, then
	$\beta > 0$ and $\mu (\tilde{D}_m)$ is summable, which proves
	\eqref{eq:summable}.
\end{proof}

We note that the sufficient condition obtained by Kelmer \cite{Kelm} in the setting of discrete time homogeneous flows acting on a finite volume quotient of $\mathbbm{H}^n$ is slightly better than what we get in our setting. More precisely, \cite[Theorem 2]{Kelm} states that in the mentioned setting
\begin{equation*}
	\sum_{j=0}^{\infty} \frac{1}{2^j\mu(B_{2^j})}<\infty \qquad\Rightarrow\qquad \mu(\E)=1.
\end{equation*} 
Inserting either $\mu (B_m) \geq \frac{c (\log m)^2}{m}$ or $\mu (B_m) \geq \frac{c}{m^a}$ in the above both result in convergent sums. A bound like $\mu (B_m) \geq \frac{c (\log m)}{m}$ or lower would be a sharper bound compared to that of Kelmer. However, the settings are very different and it is not clear what an optimal bound would look like in either setting.
\section{Application to examples}\label{section4} 

Many systems are known to have either polynomial or exponential decay
of correlation for $L^1$ against $BV$ in the sense of \eqref{expmix}
and \eqref{polmix}.
Examples of exponential decay includes $T$ being a piecewise expanding
interval map and $\mu$ being a Gibbs measure, $T$ being a quadratic
map for a Benedicks--Carleson parameter and $\mu$ being the absolutely
continuous invariant measure, \cite{liveranietal,young}.  Hence, for
these systems, $\mu(\E) = 1$ whenever $\mu (B_m) \geq \frac{c (\log m)^{2}}{m}$ for some $c>0$ sufficiently large.

Hitting time statistics is known for many interesting dynamical systems. Since HTS is often true for systems with sufficiently nice mixing properties, our necessary and sufficient conditions tend to hold for many of the same systems.
Examples of HTS in dynamics include transitive Markov chains, Axiom A diffeomorphisms, uniformly expanding maps of the interval, non-uniformly hyperbolic maps, partially hyperbolic dynamical systems and toral automorphisms. Hence, for these systems $\mu(\E)=0$ whenever $\mu(B_m)\leq \frac{c}{m}$ for some $c\in\R$. For a comprehensive overview of these results and references, see \cite[Chapter 5]{Lucarini}. Furthermore, \cite{FFT1, FFT2} establish a direct connection between HTS and so-called extreme value theory meaning that many HTS results can be obtained simply by translating known extreme value laws. See again \cite{Lucarini} for an overview of such results. 

In the following subsections we describe in more detail how Theorem~\ref{thm1}, \ref{thmMP} and \ref{ThmCF} are deduced.

\subsection{The doubling map}
HTS is known to hold for the doubling map. The precise statement is as follows. Let $r_m\in \R$ be a sequence and set $E_m=B(y,r_m)$, i.e.\ the ball with center $y\in [0,1]$ and radius $r_m$. For any sequence $r_m\to 0$ we have 
\begin{equation*}
G_{E_m} (t)=\begin{cases}
1-e^{-t} & \text{if } y \text{ is not a periodic point}\\
1-e^{-\para{1-\frac{1}{2}}^p t} & \text{if } y \text{ is a periodic point of prime period } p\in\N.
\end{cases}
\end{equation*}
This explicit result can be deduced by applying \cite[Theorem~2]{FFT1}
(along with \cite[p.7]{FFT2}) to
\cite[Corollary~4.2.11]{Lucarini}. Hence Theorem~\ref{thm1}\ref{thm1a}
follows directly from Corollary~\ref{cor:measure0RTS}.

The doubling map with Lebesgue measure is a well-known example of a
piecewise expanding interval map with a Gibbs measure. Hence it is exponentially mixing for $L^1$ against $BV$ \cite{liveranietal} and
Theorem~\ref{thm1}\ref{thm1b} follows directly from
Theorem~\ref{the:sufficient}.

\subsection{The Manneville--Pomeau map}

Here we explain how Theorem~\ref{thmMP} follows from our results. We begin with Theorem~\ref{thmMP}\ref{thmMPa}.

For the Manneville--Pomeau maps the following is known regarding
hitting time statistics in the case when $\alpha < 1$, that is when
the invariant measure $\mu_\alpha$ is finite. Let $E_m=B(y,r_m)$ as
above, then for any sequence $r_m\to 0$ we have
\begin{equation*}
G_{E_m} (t)=\begin{cases}
1-e^{-t} & \text{if } y \text{ is not a periodic point}\\
1-e^{-\bigl(1-\frac{1}{\num{Dg_{\alpha}^p (y)}}\bigr) t} & \text{if } y \text{ is a periodic point of prime period } p \\
0 & \text{if } y=0,
\end{cases}
\end{equation*}
where $Dg_{\alpha}^p (y)$ denotes the derivative of $g_{\alpha}^p$ at
the point $y$. This follows again by applying \cite[Theorem~2]{FFT1}
(along with \cite[p.7]{FFT2}) to \cite[Theorem~1 and 2]{FFTV}. Hence
Theorem~\ref{thmMP}\ref{thmMPa} follows directly from
Corollary~\ref{cor:measure0RTS}. Note that
Theorem~\ref{thmMP}\ref{thmMPa} actually also holds for balls centered
at 0 as well.

We proceed to deducing Theorem~\ref{thmMP}\ref{thmMPb}, \ref{thmMPc} and \ref{thmMPd}.
Note first that the case $\alpha=1$ in Theorem~\ref{thmMP}\ref{thmMPc} follows directly from Corollary~\ref{cor:measure0}\ref{cor1b}. 

The Manneville--Pomeau map is not known to have decay of correlations
for $L^1$ against $BV$. However, through a technique known as
\emph{inducing}, explained below, we can obtain results almost as
strong as if it had exponential decay of correlations with respect to said norms. We let $S
\colon [\frac{1}{2}, 1) \to [\frac{1}{2}, 1)$ be the first return map
    of $g_\alpha$ to the interval $[\frac{1}{2}, 1)$.  The structure
      of the map $S$ is illustrated in Figure~\ref{fig:MP}.

Given a point $x \in [\frac{1}{2}, 1)$, there is then a sequence $R_k
(x)$ such that
\[  
S^k (x) = g_\alpha^{R_k (x)} (x).
\]
The sequence $R_k (x)$ satisfies
\[
R_k (x) = \sum_{j=0}^{k-1} R (S^j (x)),
\]
where $R$ is the return time $R(x) = \min \{\, n \geq 1 : g_\alpha^n
(x) \in [\frac{1}{2}, 1) \,\}$.

\begin{figure}
	\begin{centering}
		\includegraphics[width=0.6\textwidth]{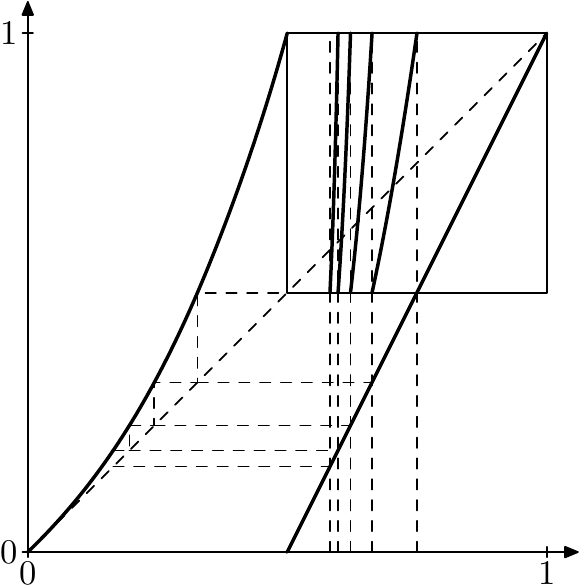}
	\end{centering}
	\caption{The Manneville--Pomeau map and its first return map.}
	\label{fig:MP}
\end{figure}

The return map $S$ is uniformly expanding and it follows by the paper
of Rychlik \cite{Rychlik} that it has
exponential decay of correlations for $L^1$ against $BV$. Hence we may
apply Theorem~\ref{the:sufficient} to $S$.

The absolutely continuous invariant measure $\mu_\alpha$ of $g_\alpha$
is finite on $[\frac{1}{2}, 1)$ and we write $\tilde{\mu}_{\alpha}$ for the normalized measure, i.e.\
  $\tilde{\mu}_{\alpha} ([\frac{1}{2}, 1)) = 1$. The measure $\mu_\alpha$ is
    finite on $[0,1)$ if and only if $\alpha \in (0,1)$. In fact, the
      density $h$ of the measure $\mu_\alpha$ is a decreasing and
      positive function, and it satisfies $h(x) \sim x^{-\alpha}$ when
      $x$ is close to zero \cite{Thaler}.



Using the first return map enables us to estimate $\tilde{\mu}_{\alpha} (\E(\mathbf{B}) \cap [\frac{1}{2},1))$ for $\mathbf{B}=\tub{B(y,r_m)}_{m=1}^{\infty}$ with center $y\in[\frac12,1)$ depending on the rate of shrinking of $r_m$. At the end of the section we argue why this is sufficient.
For now set $B_m=B(y,r_m)$ and assume that $y\in[\frac12,1)$. By using the fact that $g_{\alpha}^n(x)\in B_m$ can only happen along the subsequence $n_k:=R_k(x)$, a short argument gives the inclusions
\begin{equation}
\label{A1A2inclusion}
	A_1 \subset \E \cap \Big[\frac{1}{2},1\Big ) \subset A_2
\end{equation}
where
\begin{equation*}
A_1 = \Bigl\{ \, x
\in \Bigl[ \frac{1}{2},1 \Bigr) : \exists\,m_0\,\forall\, m\geq m_0\, \exists\, k < m : S^k (x) \in B_{R_{m+1}(x)} \, \Bigr\},
\end{equation*}
and
\begin{equation*}
 A_2 = \Bigl\{\, x \in
    \Bigl[ \frac{1}{2},1 \Bigr) : \exists\,m_0\,\forall\, m\geq m_0\, \exists\, k < m :
      S^k (x) \in B_{R_m(x)} \, \Bigr \},
\end{equation*}
The small argument verifying \eqref{A1A2inclusion} is left to the reader.

Assume first that $\mu_{\alpha}$ is finite, i.e.\ $\alpha\in (0,1)$. We will use \eqref{A1A2inclusion} together with the following
estimate of $R_n$ from
\cite[Theorem~2.19]{GalatoloHollandPerssonZhang}. There exists a constant $C > 1$ such that the set of $x$ for which the inequality,
\begin{equation}
\label{MPinequality1}
n \leq R_n (x) \leq Cn
\end{equation}
does not hold is of arbitrarily small measure if $n$ is sufficiently large.
Let $N_0$ be so large that the set of $x$ for which \eqref{MPinequality1} holds for $n>N_0$ is of measure at least $1-\delta$ and call this set $D_{\delta}$. For $x\in D_{\delta}$ we may apply \eqref{MPinequality1} and we
find that the condition $\tilde{\mu}_\alpha (B_m) \geq \frac{c (\log m)^{2 +
    \varepsilon}}{m}$ implies that $\tilde{\mu}_\alpha (B_{R_n(x)}) \geq \frac{c
  (\log n)^{2 + \varepsilon}}{C n} > \frac{c_0 (\log n)^2}{n}$ for large
$n$. So by picking $C_m:=B(y,r_m)$ with $\tilde{\mu}_\alpha(C_m)=\frac{c (\log m)^{2 +\varepsilon}}{m}$ we get $\tilde{A}_1\subset A_1$ where
\begin{equation*}
\tilde{A}_1:=\Bigl\{ \, x
\in \Bigl[ \frac{1}{2},1 \Bigr) : \exists\,m_0\,\forall\, m\geq m_0\, \exists\, k < m : S^k (x) \in C_m \, \Bigr\}\cap D_{\delta}.
\end{equation*}
The first set in the intersection is $\E(C_m)$ for the system $([\frac12,1), S, \tilde{\mu}_\alpha)$ which has measure one by Theorem~\ref{the:sufficient}. This shows that $\tilde{\mu}_\alpha(\tilde{A}_1)\geq 1-\delta$ and since $\delta$ is arbitrary we get $\tilde{\mu}_\alpha(A_1)=1$. This implies that $\tilde{\mu}_\alpha(\E\cap [\frac12,1)) = 1$.

Assume now that $\alpha> 1$ in which case $\mu_{\alpha}$ is infinite. In this case we use \eqref{A1A2inclusion} together with a different estimate of $R_n$ which also originates from \cite[Theorem~2.19]{GalatoloHollandPerssonZhang}. In this case we have that for any $\kappa > 0$, the set of $x$ for which the estimate
\begin{equation} \label{eq:Rnestimate}
n^{\alpha - \kappa} \leq R_n (x) \leq n^{\alpha + \kappa}
\end{equation}
does not hold is of arbitrarily small measure if $n$ is sufficiently large. 
Let again $N_0$ be so large that the set $D_{\delta}$ of $x$ for which \eqref{eq:Rnestimate} holds for $n>N_0$ is of measure at least $1-\delta$. For $x\in D_{\delta}$ we have that if there exists an $\varepsilon>0$ such that $\tilde{\mu}_{\alpha}(B_m)\leq \frac{c}{m^{\frac{1}{\alpha}+\varepsilon}}$ for any $c>0$, then
\begin{equation*}
	\tilde{\mu}_{\alpha}(B_{R_n})\leq \frac{c}{R_n^{\frac{1}{\alpha}+\varepsilon}}\leq \frac{c}{n n^{\varepsilon\alpha-\frac{\kappa}{\alpha}-\varepsilon\kappa}}<\frac{c_1}{n}\quad,\quad c_1<1
\end{equation*}
for sufficiently large $n $ when we choose $\kappa$ sufficiently small compared to $\varepsilon$. As before we may pick $C_m:=B(y,r_m)$ such that $\mu(C_m)=\frac{c_1}{m}$. As before we can define a set $\tilde{A_2}$ to which we can apply Corollary~\ref{cor:measure0} and by the same reasoning we conclude $\tilde{\mu}_\alpha(\E\cap [\frac12,1)) = 0$.

The proof that $\tilde{\mu}_{\alpha}(B_m)\geq \frac{c}{m^{\frac{1}{\alpha}-\varepsilon}}$ implies $\tilde{\mu}_\alpha(\complement \E\cap [\frac12,1)) = 0$ is similar to the previous two cases and we omit the details.

%
%
%

We now argue why our results for $\tilde{\mu}_{\alpha} (\E(\mathbf{B}) \cap [\frac{1}{2},1))$ imply the general statement. Note first that if $y\in [\frac{1}{2},1)$ and $\tilde{\mu}_{\alpha} (\E(\mathbf{B}) \cap [\frac{1}{2},1)) = 0$ then $\mu_{\alpha}(\complement\E(\mathbf{B}))>0$. By Lemma~\ref{Lemma1} we then have $\mu_{\alpha}(\E(\mathbf{B}))=0$. The identical argument works when $\E$ is replaced by $\complement \E$.

%
%

To go from $y\in [0,\frac12)$ to $y\in [0,1)$ requires only a little
more consideration. Notice first that each assumption on $\mu(B_m)$ in
Theorem~\ref{thmMP}\ref{thmMPb}, \ref{thmMPc} and \ref{thmMPd} is
invariant under multiplication by a constant. Assume for example the
setting of Theorem~\ref{thmMP}\ref{thmMPc}, i.e.\ $\alpha\geq 1$ and
$\mu(B_m)\leq \frac{c}{m^{\frac{1}{\alpha}+\varepsilon}}$ for some
$\varepsilon>0$ and any $c>0$. Assume that $y\in [0,\frac12)$. Let
$k_0$ denote the smallest number such that $g_{\alpha}^{k_0}(y)\in
[\frac12,1)$. There exists a $K>0$ such that for all $m\in\N$ we
have $\mu_{\alpha}(g_{\alpha}^{k_0}(B(y,r_m)))\leq
K\mu_{\alpha}(B(y,r_m))$. This is an easy consequence of
l'H\^{o}pitals rule applied to the function
$f(r):=\mu_{\alpha}(g_{\alpha}(B(y,r)))/\mu_{\alpha}(B(y,r))$. Pick
$\tilde{B}_m$ to be the smallest ball with center
$g_{\alpha}^{k_0}(y)$ such that $g_{\alpha}^{k_0}(B(y,r_m))\subset
\tilde{B}_m$. Then $\tilde{B}_m$ also satisfies the assumption of
Theorem~\ref{thmMPc} and we know that
$\mu_{\alpha}(\E(\tilde{B}_m))=0$ from the arguments above.

We argue that $\E(B_m)\subset \E(\tilde{B}_m)$. Assume that $x\in \E(B_m)$, i.e.\ $\exists\, m_0\,\forall\, m\geq m_0 \,\exists\, k<m:g_{\alpha}^k (x)\in B_m$. But if $g_{\alpha}^k (x)\in B_m$ then $g_{\alpha}^{k+k_0} (x)\in \tilde{B}_m$. Hence $\forall\, m\geq m_0+k_0\,\exists\, k<m$ such that $g_{\alpha}^{k} (x)\in \tilde{B}_m$, i.e.\ $x\in\E(\tilde{B}_m)$. The cases Theorem~\ref{thmMP}\ref{thmMPb} and \ref{thmMPd} follow by analogue arguments. This completes the proof of Theorem~\ref{thmMP}. 

\subsection{The Gau\ss{} map}
 
In this section, we will consider the Gau\ss{} map $G \colon (0,1] \to
[0,1)$ defined by $G(x) = \frac{1}{x} \mod 1$. This is a piecewise
expanding map with infinitely many branches. There is a unique measure
which is an invariant probability measure and absolutely continuous
with respect to the Lebesgue measure. We denote this measure by $\mu$,
and its density with respect to Lebesgue measure is given by $h(x) =
\frac{1}{\log 2} \frac{1}{1+x}$.

HTS for the Gau\ss{} map is known for $B_m:=[0,r_m)$, i.e.\ the interval with fixed left endpoint being 0 and right endpoint shrinking towards 0. More precisely, for $E_m=[0,r_m]$ we have, 
\begin{equation*}
G_{E_m} (t)=1-e^{-t}.
\end{equation*}
This follows from a classic result of Doeblin\footnote{To be precise,
  Doeblins proof contained a gap which was repaired by Iosifescu
  \cite{Iosifescu} in 1977, but not before Galambos \cite{Galambos}
  had proven a special case 1972 which is sufficient for our
  purposes.} \cite{Doeblin} which may easily be translated into the
above, see for example \cite[Section 5]{Haydn}. Hence for this choice
of targets, Corollary~\ref{cor:measure0RTS} applies. In particular, we
have $\mu (\E) = 0$ provided $\mu([0,r_m)) \leq \frac{c_0}{m}$ for some
  constant $c_0 > 0$.

In order to show that we may apply Theorem~\ref{the:sufficient} we
need the following short argument. Define $g:[0,1]\to [0,1]$ by $g(x)
= \frac{1}{|G'(x)|} = x^2$ whenever $x$ is not 0 or of the form $x =
\frac{1}{n}$ for some $n \in \N$. In the remaining points, we let
$g(x) = 0$. Obviously, $g$ is then of bounded variation, since the
heights of the jumps at $\frac{1}{n}$ are summable.

We have $\lVert g \rVert_\infty = 1$, but $\lVert g \cdot g \circ G
\rVert_\infty < 1$, since $g (x) \to 1$ only if $x \to 1$, and $x=1$
is not a fixed point under $G$. This implies that $G^2$ together with
$g_2 = g \cdot g \circ G$ satisfies the assumptions of Rychlik
\cite{Rychlik}. Hence we may conclude \cite[Theorem~1]{Rychlik} that
if $P \colon BV \to BV$ denotes the transfer operator associated to
$G$, then, since $G$ is mixing, $P$ can be written in the form $P = Q
+ R$, where $Q$ is the projection on the invariant density $h$, $R$
has a spectral radius which is strictly less than $1$ and $QR = RQ =
0$.

Let $f \in L^1$ and $g \in BV$. Let $c = - \int g \, \mathrm{d} \mu$,
so that $\int (g + c) \, \mathrm{d} \mu = 0$. Then
\[
\int f \circ G^n g \, \mathrm{d} \mu = \int f \circ G^n (g + c) \,
\mathrm{d} \mu + \int f \, \mathrm{d} \mu \int g \, \mathrm{d} \mu.
\]

By the choice of $c$, we have $Q ( (g+c) h) = 0$. Hence
\[
\int f \circ G^n \cdot (g + c) \, \mathrm{d} \mu = \int f \cdot R^n (
(g+c) h ) \, \mathrm{d} x.
\]
Since $R$ has a spectral radius strictly less than $1$, there are
positive constants $C_1$ and $\tau$ such that $\lVert R^n ( (g+c) h )
\rVert_{BV} \leq C e^{- \tau n} \lVert (g+c) h \rVert_{BV}$. In
particular
\begin{align*}
\lVert R^n ( (g+c) h ) \rVert_\infty & \leq \lVert R^n ( (g+c) h )
\rVert_{BV} \leq C_1 e^{- \tau n} \lVert (g+c) h \rVert_{BV}
\\ &\leq C_1 e^{- \tau n} ( \lVert g h \rVert_{BV} + \lVert c h
\rVert_{BV} ) \leq C_2 e^{- \tau n} \lVert g \rVert_{BV}.
\end{align*}
From this, it follows that
\begin{align*}
\biggl| \int f \cdot R^n ( (g+c) h ) \, \mathrm{d} x \biggr| &\leq
\int |f| \, \mathrm{d} x \cdot \lVert R^n ( (g+c) h ) \rVert_\infty
\\& \leq \frac{\lVert f \rVert_1}{2\log 2} \cdot C_2 e^{-\tau n}
\lVert g \rVert_{BV}.
\end{align*}
Hence, with $C = \frac{C_2}{2\log 2}$ we have
\[
\biggl| \int f \circ G^n g \, \mathrm{d}\mu - \int f \, \mathrm{d} \mu
\int g \, \mathrm{d} \mu \biggr| \leq C e^{- \tau n} \lVert f \rVert_1
\lVert g \rVert_{BV}.
\]

In conclusion, we may apply Theorem~\ref{the:sufficient} to the
Gau\ss{} map. In particular, we have $\mu (\E) = 1$ provided $\mu
(B_m) \geq \frac{c_0 (\log m)^2}{m}$ for some constant $c_0 >
\tau^{-1}$.

In order to obtain Theorem~\ref{ThmCF} we now use the fact that when
$r_m$ is small and $G^{j-1}(x)\in [0,r_m)$ then
  $a_j(x)\sim\frac{1}{G^{j-1}(x)}$. An easy calculation shows that
  in the definition of $\E([0,r_m))$ for the Gau\ss{} map we can
    replace $\tub{G^{k}(x)}_{k=0}^{m-1}(x)\cap [0,r_m)\neq\emptyset$
      with $\tub{a_k(x)}_{k=1}^m\cap
      \left[\frac{1}{r_m},\infty\right)\neq \emptyset$. The only thing
        left to do now is to compute the bounds on $r_m$ when
        $\mu([0,r_m)) \leq \frac{c_0}{m}$ and $\mu([0,r_m)) \geq \frac{c_0
              (\log m)^2}{m}$. Using the density of the Gau\ss{}
            measure we get that
\begin{equation*}
  \mu([0,r_m))=\frac{1}{\log 2}\log(1+r_m).
\end{equation*}
This leads to the following conclusions. If $r_m \leq
e^{\frac{c_0 \log 2}{m}} - 1$ for some $c_0 > 0$, then
\begin{equation} \label{eq:measurezero}
  \mu \Bigl( \Bigl\{\, x\in [0,1]:\exists\,m_0\in\N\,\forall m\geq
  m_0:\max_{1\leq k\leq m} a_k(x)\geq r_m^{-1} \, \Bigr\} \Bigr) =0,
\end{equation}
and if $r_m \geq e^{\frac{c_0 \log 2 (\log m)^2}{m}} - 1$ for some $c_0 > 0$
sufficiently large, then
\begin{equation} \label{eq:measureone}
  \mu \Bigl( \Bigl\{\, x\in [0,1]:\exists\,m_0\in\N\,\forall m\geq
  m_0:\max_{1\leq k\leq m} a_k(x) \geq r_m^{-1} \,\Bigr\} \Bigr) =1.
\end{equation}

Since $\frac{c_0 \log 2}{m} \leq e^{\frac{c_0 \log 2}{m}} - 1 \leq
\frac{2c_0 \log 2}{m}$ if $m$ is large and since $c_0 > 0$ arbitrary,
we can conclude that \eqref{eq:measurezero} holds provided $r_m \leq
\frac{1}{cm}$ for some $c > 0$. Letting $b_m = r_m^{-1}$, we have that
\[
  \mu \Bigl( \Bigl\{\, x\in [0,1]:\exists\,m_0\in\N\,\forall m\geq
  m_0:\max_{1\leq k\leq m} a_k(x)\geq b_m \, \Bigr\} \Bigr) =0,
\]
if $b_m \geq cm$ for some $c > 0$. This is statement
Theorem~\ref{ThmCF}\ref{thmCFa}.

Similarly, we may conclude that \eqref{eq:measureone} holds provided
$r_m \geq \frac{(\log m)^2}{cm}$ for some sufficiently small $c >
0$. Letting $b_m = r_m^{-1}$ we then have
\[
  \mu \Bigl( \Bigl\{\, x\in [0,1]:\exists\,m_0\in\N\,\forall m\geq
  m_0:\max_{1\leq k\leq m} a_k(x) \geq b_m \,\Bigr\} \Bigr) =1,
\]
if $b_m \leq \frac{c m}{(\log m)^2}$ for some sufficiently small $c >
0$. This is Theorem~\ref{ThmCF}\ref{thmCFb}.

\section{Results on the Hausdorff and packing dimension}
\label{section5}

In this section we consider the case when $T$ is the doubling map $[0,1) \to [0,1)$,
defined by $T(x) = 2x \mod 1$, and $\mu$ is Lebesgue measure.  If
we put $\Sigma = \{0,1\}^\N$, then every $x \in [0,1)$ can
be coded by a sequence $(x_i) = (d_i(x))_{i=0}^\infty$ such that
\[
d_i (x) = \left\{ \begin{array}{ll} 0 & \text{if } T^i (x) \in
[0,\frac{1}{2}), \\ 1 & \text{if } T^i (x) \in
[\frac{1}{2},1). \end{array} \right.
\] 
Note that in the symbolic setting the doubling map becomes the left shift on $\Sigma$, i.e.\ for $d(x)=x_0,x_1,\dots$ we have $T(d(x))=x_1,x_2,\dots$. Given a finite sequence $x_0, x_1, \ldots, x_m$, we let
\[
C (x_0, x_1, \ldots, x_m) = \{\, x \in [0,1) : d_i (x) = x_i \text{ for
} n = 0,1,\ldots,m \,\}.
\]
Then $C (x_0, x_1, \ldots, x_m)$ is an interval of length $2^{-(m+1)}$.

We recall the construction of the packing dimension. Let $F\subset\R^d$ and let $\delta>0$. A collection $\tub{B_i}$ of disjoint balls of radii at most $\delta$ with centres in $F$ is called a $\delta$-\emph{packing collection} for $F$. For $s\geq 0$, let
\begin{equation*}
	\tyk{P}_{\delta}^s(F)=\sup\tub{\sum_i \num{B_i}^s : \tub{B_i} \text{ is a } \delta\text{-packing collection for } F}
\end{equation*}
and
\begin{equation*}
	\tyk{P}_0^s(F)=\lim_{\delta\to 0} \tyk{P}_{\delta}^s(F).
\end{equation*}
We then define the $s$-dimensional \emph{packing measure} by
\begin{equation*}
	\tyk{P}^s(F)=\inf\tub{\sum_i \tyk{P}_0^s(F_i) : F\subset \bigcup_{i}F_i}
\end{equation*}
where the infimum is taken over all countable covers of $F$. Finally, the \emph{packing dimension} of $F$ is defined by
\begin{equation*}
	\dimp F=\sup{\tub{s:\tyk{P}^s(F)=\infty}}=\inf{\tub{s:\tyk{P}^s(F)=0}}.
\end{equation*}
In the hierarchy of dimensions the packing dimension falls between the Hausdorff dimension and the upper box-counting dimension in the sense that
\begin{equation*}
	\dimh F\leq \dimp F\leq \overline{\textup{dim}}_{B} F.
\end{equation*} 
\begin{theorem}
	Suppose $B_m = B(0, 2^{-sm})$, where $s \in (0,1)$. Then,
	\[
	\dimp \E = 1 - s.
	\]
	If $s\geq 1$ then $\dimp \E = 0$. Indeed, $\E$ is a countable set in this case.
\end{theorem}

\begin{proof}
	If $s = 1$, then we will prove that $\E$
	consists only of those $x \in [0,1)$, such that $T^n (x) = 0$ for
	some $n$. Hence $\E$ is the set of all finite words concatenated with an infinite tail of zeroes from which it is clear that $\E$ is countable. Then $\E$ must also be countable for all $s>1$. 
	Assume $x \in \E$. Then $x \in A_{n}$ for
	some $n$ which means that $\tub{T^i x}_{i=0}^{m-1}\cap B_m\neq\emptyset$ for all $m\geq n$. Then somewhere from digit number 0 to digit number $m-1$, a
	block of $m$ zeroes starts. Regardless of where this block starts it will overlap with the digit at place $m-1$ and hence the digit on place $m-1$ is $0$. Since this is true for all $m\geq n$ we have shown that $d_m (x) = 0$ for all $m\geq n-1$.
%

	
	From now on, we assume that $s < 1$. We will first prove that
        $\dimp \E \leq 1 - s$. Since $\dimp$ is countably stable, $\E
        = \bigcup A_n$ and $\dimp\leq \dimb$, it is enough to prove
        that $\dimb A_n \leq 1 - s$ for all $n$, where $\dimb$ denotes
        the box dimension.  If there is an $n$ such that $T^n (x) =
        0$, then $x \in A_n$. As in the introduction, let $\Lambda$ denote the set of all such
        points which, as discussed above, is countable and may be
        disregarded, since the packing dimension of any countable set
        is 0. Set $A_n':=A_n\backslash \Lambda$.
	
	We first prove that $x\in A_n'$ if and only if
        $(d_i(x))_{i=0}^{\infty}$ has blocks of $n_j+1$ zeroes
        starting at position $m_{j-1}$ for all $j\geq 1$ where
        $(n_j)_{j=0}^{\infty}$ and $(m_j)_{j=0}^{\infty}$ are strictly
        increasing sequences satisfying
	\begin{enumerate}[label=\alph*)]
		\item\label{a} $m_0<n$.
		\item\label{b} $n_j\geq sm_j$.
		\item\label{c} $m_{j-1}<(1-s)m_j$.
	\end{enumerate}
	Suppose first that $x \in A_n'$.
	Then starting somewhere not later than at position $n-1$, the
        sequence $(d_i (x))_{i=0}^\infty$ contains a block of at least
        $s n$ zeroes\footnote{Since $sn$ is typically non-integer we
          should, in principle, be more diligent and write $\lceil sn
          \rceil$. However, to improve readability we let the relevant
          ceiling and floor functions be implicitly understood
          throughout the proof. The outcome is invariant under this
          abuse of notation.}. We let $n_1$ be the length of the
        largest such block, and let $m_0<n$ be the position of its
        first zero. This satisfies \ref{a}.
	
	The above property is true for all $m\geq n$, i.e.\ starting
        somewhere not later than at position $m-1$ a block of at least
        $s m$ zeroes start. This allows us to define sequences
        $(n_j)_{j=1}^\infty$ and $(m_j)_{j=0}^\infty$ as follows. The
        numbers $n_1$ and $m_0$ are already defined. Suppose that
        $n_j$ and $m_{j-1}$ are defined. Then we let $m_j$ be the
        position of the first digit of the leftmost block of at least
        $n_j + 1$ consequtive zeroes in the sequence $(d_i
        (x))_{i=0}^\infty$ and we let $n_{j+1}$ be the maximal number
        of zeroes in such a block.  Note that, if we make any change
        in the digits $d_i (x)$ outside of the blocks of $n_j$ zeroes
        described above, then we get a new sequence $d_i (y)$ with $y
        \in A_n'$.
	
	Since $x \in A_n'$ and $n_j$ was chosen to be maximal, the
        block of length $n_j$ must be long enough to ensure that
        $\tub{T^i x}_{i=0}^{m_j-1}\cap B_{m_j}\neq \emptyset$,
        i.e.\ we always have $n_j \geq s m_j$ and \ref{b} is
        satisfied. Furthermore, again due to $n_j$ being maximal and since $x\notin \Lambda$, the
        blocks of zeroes are separated, and we have $m_{j-1} + n_j <
        m_j$. Hence $m_{j-1} < (1 - s) m_j$ and \ref{c} is satisfied.
	
	
  	Conversely it is clear that any $x\in [0,1]$ for which \ref{a}, \ref{b} and \ref{c} holds true for $(d_i (x))_{i=0}^\infty$ is an element of $A_n'$. 
	
	Let now $N$ be fixed and take $k$ such that $m_{k-1} < N \leq m_k$.
	From $m_{j-1} < (1 - s) m_j$ we obtain that
	\[
	1 < m_1 < (1-s)^{k-2} m_{k-1} < (1-s)^{k-2} N.
	\]
	Hence
	\[
	k < 2 + \frac{\log N}{- \log (1-s)}.
	\]
	
	The numbers of zeroes in $(d_i(x))_{i=0}^\infty$ between digit
	number $m_0$ and digit number $N$ is at least $s (N - m_0)$. This is
	the case, since the finite sequence $d_{m_0} (x), \ldots, d_N (x)$
	can be cut into $k$ sequences starting at $d_{m_j} (x)$, $j = 0,1,
	\ldots k-1$ and each of these subsequences contains at least a
	quotient of $s$ zeroes.
	
	We will now cover the set $A_n'$ by intervals of the form
	\[
	C (x_0, x_1, \ldots x_N).
	\]
	The sequences $x_0, x_1, \ldots, x_N$ that we need to consider are
	only those that can be obtained from a sequence $(n_j)_{j=1}^k$ and
	$(m_j)_{j=0}^k$ satisfying the inequalities \ref{a}, \ref{b} and \ref{c} and therefore
	also with $k \leq 2 + \frac{\log N}{-\log (1-s)}$.
	
	The sequence $(m_j)_{j=0}^k$ can be chosen in at most $N^{k+1}$
	different ways, and once $(m_j)_{j=0}^k$ is chosen, we can choose
	$(n_j)_{j=1}^k$ in at most $N^k$ different way, hence in total at
	most $N^{2k + 1}$ different ways to choose the sequences. (These are
	very rough estimates, but sufficient for our purpose.)
	
	Once the sequences $(n_j)_{j=1}^k$ and $(m_j)_{j=0}^k$ are chosen,
	we have specified a certain number of zeroes, while the other digits
	in the sequence $x_0, x_1, \ldots, x_N$ remain free. There are at
	most $m_0 + (1 -s) (N - m_0) \leq (1 - s)N + n$ digits that are
	free, and hence once the sequences $(n_j)_{j=1}^k$ and
	$(m_j)_{j=0}^k$ are chosen, we may choose the sequence $x_0, x_1,
	\ldots, x_N$ in at most $2^{(1-s)N + n}$ ways.
	
	In total, the number of sequences $x_0, x_1, \ldots, x_N$ that we
	need in order to cover $A_n'$ with the sets $C (x_0, x_1, \ldots
	x_N)$, are not more than
	\[
	N^{2k + 1} 2^{(1-s)N + n} \leq N^{5 + 2 \frac{\log N}{- \log (1-s)}}
	2^{(1-s)N + n}.
	\]
	Since the sets $C (x_0, x_1, \ldots x_N)$ have diameter $2^{-(N+1)}$ we get from the definition of box dimension that
	\begin{equation*}
		\dimb(A_n')=\lim_{N\to\infty} \frac{\log\para{N^{5 + 2 \frac{\log N}{- \log (1-s)}}
				2^{(1-s)N + n}}}{-\log\para{2^{-(N+1)}}}=1-s.
	\end{equation*}
	The last equality follows since $N^{\log N}$ grows with $N$ slower than any exponential.

	We now finish by proving that $\dimp \E \geq 1 - s$. Let
        $(m_j)_{j=1}^\infty$ be a strictly increasing sequence of
        natural numbers. Using this sequence, we will construct a
        subset of $\E$ and prove that the packing dimension of this
        subset is $1 - s$ if the sequence $(m_j)_{j=1}^\infty$ is
        chosen such that 
        \begin{equation} \label{eq:mk}
          \lim_{k \to \infty} \frac{\sum_{j=1}^{k} m_j}{m_k} = 1.
        \end{equation}
	For example, we could choose the sequence inductively as $m_k:=k\sum_{j=1}^{k-1} m_j$, however, the explicit choice is irrelevant.
	Let $F$ consist of those $x \in [0,1)$ such
	that
	\[
	d_{m_j} (x), d_{m_j+1} (x), \ldots, d_{m_j + s m_{j+1}} (x) = 0,0,
	\ldots, 0
	\]
	for all $j\geq 0$. Then $F \subset \E$.
	
	We let $\mu$ be the probability measure on $F$ defined by
	\[
	\mu (C(x_0, x_1, \ldots, x_n)) = \frac{1}{N(n)},
	\]
	if $C(x_0, x_1, \ldots, x_n)$ intersects $F$ where $N(n)$ is
        the number of intervals $C(x_0, x_1, \ldots, x_n)$ that have
        non-empty intersection with $F$. Otherwise we assign measure
        0. The upper pointwise dimension of $\mu$ at $x \in F$ is
        defined by
	\[
	\overline{d}_\mu (x) =\limsup_{r \to 0} \frac{\log \mu
          \left(B(x,r)\right)}{\log r}.
        \]
        Let $r_n = 2^{-n-2}$. Then $B (x,r_n)$ is contained in the
        cylinder $C(x_0, x_1, \ldots, x_n)$ and one of the neigbouring
        cylinders. The measure of the neighbouring cylinder is either
        zero or equal to that of the cylinder $C(x_0, x_1, \ldots,
        x_n)$. Hence
        \[
        \mu (B(x,r_n)) \leq 2 \mu (C(x_0, x_1, \ldots, x_n)),
        \]
        and since $\log r_n < 0$ we have
        \[
        \frac{ \mu (B(x,r_n))}{\log r_n} \geq \frac{ \log (2 \mu (C(x_0,
          x_1, \ldots, x_n)))}{\log r_n}.
        \]
        It therefore follows that
        \begin{align*}
	\overline{d}_\mu (x) &= \limsup_{r \to 0} \frac{\log \mu
          \left(B(x,r)\right)}{\log r} \geq \limsup_{n \to \infty} \frac{\log \mu
          \left(B(x,r_n)\right)}{\log r_n} \\ &\geq \limsup_{n \to \infty}
        \frac{\log\mu\para{C(x_0, x_1, \ldots, x_n)}}{\log
          2^{-(n+2)}} \\
        &\geq \limsup_{k \to \infty}
        \frac{\log\mu\para{C(x_0, x_1, \ldots, x_{m_k})}}{\log
          2^{-(m_k+2)}}.
	\end{align*}
        We will use this inequality to show that $\overline{d}_\mu (x)
        \geq 1 - s$ for any $x \in F$. By
        \cite[Proposition~10.1]{falconertechniques}, this implies that
        $\dimp F \geq 1 - s$.
	
        We have that
	\[
	\mu (C (d_0(x), d_1(x), \ldots, d_{m_k} (x))) = \frac{1}{N(m_k)},
	\]
	and
	\[
	N(m_k) = 2^{m_k - s \sum_{j=1}^{k} m_{j}}.
	\]
	Hence
	\[
	\overline{d}_\mu (x) \geq \lim_{k \to \infty} \frac{m_k - s
		\sum_{j=1}^{k} m_{j}}{m_k+2} = \lim_{k \to \infty} 1 - s
	\frac{\sum_{j=1}^{k} m_j}{m_k+2}.
	\]
        Since $m_k$ are chosen so that \eqref{eq:mk} holds, we have
        \[
        \overline{d}_\mu (x) = 1 - s,
        \]
        which implies that $\dimp \E \geq \dimp F
        \geq 1 - s$.
\end{proof}

\end{document}